\documentclass[11pt]{article}
\usepackage[top=2.54cm, bottom=2.54cm, left=2.5cm, right=2.5cm]{geometry}
\usepackage[tbtags]{amsmath}
\usepackage{amssymb}
\usepackage{amsthm}
\usepackage{appendix}
\usepackage{fancyhdr}
\usepackage{latexsym}
\usepackage{mathrsfs}
\usepackage{xcolor}
\usepackage{wasysym}
\usepackage{graphicx}    
\usepackage{subcaption}  
\usepackage{fancyhdr}
\usepackage[colorlinks=true, linkcolor=blue, citecolor=cyan, urlcolor=magenta]{hyperref}
\usepackage{float}
\usepackage{graphicx}
\usepackage[numbers,sort&compress]{natbib}
\usepackage{pict2e,color,xcolor}
\usepackage{tikz}
\usepackage{tikz-3dplot}
\usepackage{tkz-graph}
\usetikzlibrary{arrows,shapes,chains}
\usepackage{enumerate}
\usepackage {verbatim}
\usepackage{amsthm}

\allowdisplaybreaks[4]

\newtheorem{theorem}{\bf Theorem}[section]
\newtheorem{lemma}[theorem]{Lemma}
\newtheorem{claim}{\indent Claim}[theorem]

\newtheorem{conjecture}[theorem]{Conjecture}
\theoremstyle{plain}

\theoremstyle{definition}
\newtheorem{definition}[theorem]{Definition}

\newtheorem{case}{\indent Case}

\usepackage{nomencl}

\usepackage{authblk}

\usepackage{etoolbox}
\AtBeginEnvironment{abstract}{\setlength{\parskip}{0.5em}}
\makenomenclature

\title{\bf On the pancyclicity of $2$-connected $[5,3]$-graphs}
\author[1]{\bf Feng Liu\footnote{Email: liufeng0609@126.com.}}
\author[2]{\bf Hongxi Liu\footnote{Email: hongxiliu1@163.com (corresponding author).}}

\affil[1]{\footnotesize School of Mathematics Sciences, Shanghai Jiao Tong University, 800 Dongchuan Road, Shanghai, 200240, China}
\affil[2]{ \footnotesize Department of Mathematics, East China Normal University, Shanghai, 200241, China}
\date{}

\begin{document}

\maketitle

\begin{abstract}

A graph $G$ is called an $[s,t]$-graph if any induced subgraph of $G$ of order $s$ has size at least $t$. In 2024, Zhan conjectured that every $2$-connected $[p + 2, p]$-graph of order at least $2p + 3$ and with minimum degree at least $p$ is pancyclic, where $p$ is an integer with $3 \leq p \leq 5$. In this paper, we confirm the conjecture for the case $p=3$, thereby taking the first step toward a complete resolution of the conjecture.

			
\end{abstract}


\section{Introduction}

A Hamiltonian cycle in a graph $G$ is a cycle containing every vertex of $G$, and $G$ is said to be Hamiltonian if it contains such a cycle. Hamiltonicity is a central topic in combinatorics, yet deciding whether a given graph is Hamiltonian is NP-complete and no easily verifiable necessary and sufficient condition is known. It is therefore natural and important to study sufficient conditions for Hamiltonicity.  The most classical such condition is Dirac's theorem \cite{dirac1952some}, which states that every graph of order $n$ with minimum degree at least $n/2$ contains a Hamiltonian cycle. Dirac's result initiated the study of degree-based sufficient conditions, and many generalizations have since been obtained; see \cite{ajtai1985first,chvatal1972note,kuhn2013hamilton,krivelevich2011critical,krivelevich2014robust,ferber2018counting,cuckler2009hamiltonian,posa1976hamiltonian} and the survey \cite{gould2014recent}.

\medskip

Pancyclicity is closely related to Hamiltonicity. A graph of order $n$ is pancyclic if it contains cycles of all lengths from $3$ to $n$; thus every pancyclic graph is Hamiltonian, making pancyclicity a stronger notion.  Bondy~\cite{bondy10pancyclic} proposed a meta-conjecture asserting that nearly every nontrivial sufficient condition for Hamiltonicity also yields pancyclicity, apart from a few simple exceptional graphs. For example, it strengthened Dirac’s theorem by proving that a graph of order $n$ with minimum degree at least $n/2$ is either pancyclic or the complete bipartite graph $K_{n/2,n/2}$~\cite{bondy1971pancyclic}. Bauer and Schmeichel~\cite{bauer1990hamiltonian}, building on earlier work~\cite{schmeichel1988cycle}, showed that conditions of Bondy~\cite{bondy1980longest}, Chvátal~\cite{chvatal1972hamilton}, and Fan~\cite{fan1984new} also imply pancyclicity, again with only a small set of exceptions.  Let $\alpha(G)$ and $\kappa(G)$ denote the independence number and connectivity of a graph $G$. Chvátal and Erdős~\cite{chvatal1972note} proved that $G$ is Hamiltonian whenever $\kappa(G)\ge \alpha(G)$; pancyclic analogues of this result have been widely studied~\cite{amar1991pancyclism}. Erd\H{o}s conjectured in 1972 that every Hamiltonian graph with $\alpha(G)\le k$ and $n=\Omega(k^2)$ vertices is pancyclic. This was recently confirmed by Dragani\'c, Correia, and Sudakov~\cite{DCS2024}, who showed that $n=(2+o(1))k^2$ already suffices.

In particular, Jackson and Ordaz~\cite{jackson1990chvatal} conjectured in 1990 that a graph $G$ is pancyclic whenever $\kappa(G)>\alpha(G)$, which, if true, would confirm Bondy’s meta-conjecture in this classical setting. Using an earlier result of Erd\H{o}s~\cite{erdos1972some}, one can show that pancyclicity holds when $\kappa(G)$ is sufficiently large as a function of $\alpha(G)$. The first explicit linear bound was obtained by Keevash and Sudakov~\cite{keevash2010pancyclicity} in 2010, who proved that $\kappa(G)\ge 600\,\alpha(G)$ is sufficient. More recently, Dragani\'c, Correia, and Sudakov~\cite{Draganic2024pancyclic} resolved the conjecture of Jackson and Ordaz asymptotically, showing that $\kappa(G)>(1+o(1))\alpha(G)$ already guarantees pancyclicity.

In the following, the concept of an $[s,t]$-graph naturally extends the notion of independence number.

\begin{definition}
Let $s$ and $t$ be given integers. A graph $G$ is called an {\it $[s,t]$-graph} if every induced subgraph of $G$ of order $s$ has size at least $t$.
\end{definition}

Two simple facts follow:
(1) every $[s,t]$-graph is also an $[s+1,t+1]$-graph;
(2) $\alpha(G)\le k$ if and only if $G$ is a $[k+1,1]$-graph.
 In 2005, Liu and Wang~\cite{LiuWang2005} proved that every 2-connected $[4,2]$-graph of order at least 6 is hamiltonian.
Later, in 2008, Liu~\cite{Liu2008} studied the pancyclicity of $[4,2]$-graphs and obtained the following result:

\begin{theorem}[Liu~\cite{Liu2008}]\label{liu42}
Every $2$-connected $[4, 2]$-graph of order at least $7$ is pancyclic.
\end{theorem}
 An alternative proof of Theorem~\ref{liu42} was presented in \cite{Zhan2024}.
In $2006$, Li and Wang \cite{LiWang2006} showed that every $2$-connected $[5,3]$-graph of order at least $8$ is hamiltonian.
More recently, in 2024, Zhan~\cite{Zhan20241} proposed the following conjecture.

\begin{conjecture}[Zhan~\cite{Zhan20241}]
Let $p \in \{3, 4, 5\}$. Every $2$-connected $[p+2, p]$-graph of order at least $2p + 3$ with minimum degree at least $p$ is pancyclic.
\end{conjecture}

Building on these developments, we verify Zhan’s conjecture for the case $p=3$ and establish the following result:

\begin{theorem}\label{thm:5-3}
Every $2$-connected $[5, 3]$-graph of order at least $9$ with minimum degree at least $3$ is pancyclic.
\end{theorem}

This theorem extends the known results on $[4,2]$-graphs and $[5,3]$-graphs,
providing further evidence of the close relationship between $[s,t]$-conditions and pancyclicity.  The proof of Theorem~\ref{thm:5-3} builds on the ideas of Draganić, Correia, and Sudakov~\cite{Draganic2024}, which were specifically developed for extending cycles in graphs.

We organize the remainder of this paper as follows: Section \ref{sec2} shows some necessary notations and definitions.
In Section \ref{sec:auxi-lemmas}, we list some useful results needed in later proofs.
Section \ref{sec4} is devoted to the proof of Theorem \ref{thm:5-3}. Finally, we give some concluding remarks in Section \ref{sec5}.

\section{Notations}\label{sec2}

For any undefined terminology or notation, we refer to the books \cite{Bondy2008,West1996}.
Denote by $V(G)$ and $E(G)$ the vertex
set and edge set of a graph $G$, respectively; and denoted by $|V(G)|$ and $|E(G)|$ the number of vertices and edges of $G$, respectively.
For a vertex $x \in V(G)$, we denote its neighborhood and degree by $N_G(x)$ and $d_G(x)$, respectively,
and let $N_G[x] = N_G(x) \cup \{x\}$ be its closed neighborhood.
When the graph is clear from context, we omit the subscript and simply write $N(x)$, $d(x)$, and $N[x]$.
The minimum degree of $G$ is denoted by $\delta(G)$.
For a vertex subset $S \subseteq V(G)$, let $N_S(x)$ (resp. $N_S[x]$) denote the neighbors of $x$ in $S$ (resp. together with $x$).
If $H$ is a subgraph of $G$, we write $N_H(x)$ (resp. $N_H[x]$) for $N_{V(H)}(x)$ (resp. $N_{V(H)}[x]$). We denote by $G[S]$ the subgraph of $G$ induced by $S$, and $G-S = G[V(G) \setminus S]$.
For an induced subgraph $H$ of $G$, let $\overline{H} = G[V(G) \setminus V(H)]$.

\medskip

Given two vertex subsets $S$ and $T$ of $G$, we denote by $[S,T]$ the set of edges having one endpoint in $S$ and the other in $T$ of $G$. An {\it $(S,T)$-path} is a path which starts at a
vertex of $S,$ ends at a vertex of $T,$ and whose internal vertices belong to neither $S$ nor $T$.
 Let $P$ be a path. We  use $P[u,v]$ to denote the subpath of $P$ between two  vertices $u$  and $v$.    Let $P$ be an oriented $(u,v)$-path. For $x\in V(P)$ with $x\ne v$, denote by $x^+$ the immediate successor on $P$.  For $x\in V(P)$  with $x\ne u$, denote by $x^-$ the predecessor on $P$.
For $S\subseteq V(P)$, let $S^+=\{x^+:x\in S\setminus\{v\}\}$ and $S^-=\{x^-:x\in S\setminus\{u\}\}$. Specifically, if $u$ has a neighbor $x$ outside $V(P)$, set $u^-=x$; if $v$ has a neighbor $z$ outside $V(P)$, set $v^+=z$.
For $x,y\in V(P)$,  $\overrightarrow{P}[x,y]$ denotes the segment of $P$ from $x$ to $y$ which follows the orientation of $P$, while $\overleftarrow{P}[x,y]$ denotes the opposite segment of $P$ from $x$ to $y$. Particularly, if $x=y$, then $\overrightarrow{P}[x,y]=\overleftarrow{P}[x,y]=x$. For $u,v\in V(G)$, we simply write $u\sim v$ if $uv \in E(G)$, and write $u\nsim v$ if $uv \notin E(G)$.

\medskip

For a positive integer $k,$ the symbol $[k]$
used in this article represents the set $\{1,2,\dots,k\}.$ Furthermore, for integers $a$ and $b$ with $a\le b,$ we
use $[a, b]$ to denote the set of those integers $c$ satisfying $a\le c\le b.$ Denote by $C_n$ and $K_n$
the cycle of order $n$ and  the complete graph of order $n$, respectively. A graph $G$ is {\it $H$-free} if $G$ has no subgraph (not necessarily induced) isomorphic to $H$. For a family $\mathcal{H}$ of graphs, $G$ is {\it $\mathcal{H}$-free} if $G$ is $H$-free for every $H\in \mathcal{H}$.

\medskip

\section{Preliminaries}\label{sec:auxi-lemmas}
In this section, we first introduce some definitions used throughout the paper,
followed by several useful lemmas that will be applied in our proofs.
 We will need the following definition of a cycle which has one triangle attached to one of its edges.

\begin{definition}[Dragani\'{c}-Correia-Sudakov \cite{Draganic2024}]
The graph $\widetilde{C}_{\ell}$ is defined as a ${C}_{\ell}$ with an additional vertex that is adjacent to two consecutive vertices on the cycle, thereby forming a triangle with them.
\end{definition}

\begin{figure}[ht]
  \begin{center}
    \begin{tikzpicture}
      \tikzset{std node fill/.style={
          draw=black, circle, fill=black,
          line width=0.6pt, inner sep=1pt, minimum size=4.5pt
      }}

      \node[std node fill] (x1) at (0,-2.5) {};
      \node[std node fill] (x2) at (1,-3) {};
      \node[std node fill] (x4) at (1,-2) {};
      \node[std node fill] (x5) at (2,-2.5) {};
      \node[xshift=6pt] at (0.8,-3.5)  {$(a)$ $\widetilde{C}_3$};
      \foreach \i/\j in {1/2,1/4,2/4,2/5,4/5}{\draw[black,line width=1pt] (x\i)--(x\j);}

      \node[std node fill] (w1) at (4,-2.5) {};
      \node[std node fill] (w2) at (5,-3) {};
      \node[std node fill] (w3) at (5,-2) {};
      \node[std node fill] (w4) at (6,-3) {};
      \node[std node fill] (w5) at (6,-2) {};
      \node[xshift=6pt] at (5,-3.5)  {$(b)$ $\widetilde{C}_4$};
      \foreach \i/\j in {1/2,1/3,2/3,3/5,4/5,2/4}{\draw[black,line width=1pt] (w\i)--(w\j);}

      \node[std node fill] (y1) at (8,-3) {};
      \node[std node fill] (y2) at (8,-2) {};
      \node[std node fill] (y3) at (9,-2.5) {};
      \node[std node fill] (y4) at (10,-3) {};
      \node[std node fill] (y5) at (10,-2) {};
      \node[xshift=6pt] at (8.8,-3.5)  {$(c)$ bowtie};
      \foreach \i/\j in {1/2,1/3,2/3,3/5,3/4,4/5}{\draw[black,line width=1pt] (y\i)--(y\j);}

      \node[std node fill] (z1) at (12,-3) {};
      \node[std node fill] (z2) at (12,-2) {};
      \node[std node fill] (z3) at (13,-2.5) {};
      \node[std node fill] (z4) at (14,-2.5) {};
      \node[std node fill] (z5) at (15,-3) {};
      \node[std node fill] (z6) at (15,-2) {};
      \node[xshift=6pt] at (13.3,-3.5)  {$(d)$ 1-dumbbell};
      \foreach \i/\j in {1/2,1/3,2/3,3/4,4/5,4/6,5/6}{\draw[black,line width=1pt] (z\i)--(z\j);}
    \end{tikzpicture}
    \caption{\small {Illustrations of the configurations.}}
    \label{Figureeight}
  \end{center}
\end{figure}
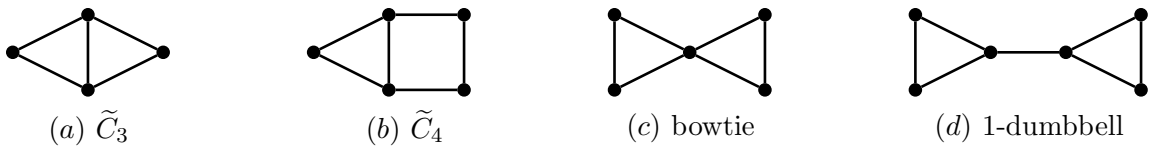

See Figure~\ref{Figureeight}$(a)$ and $(b)$ for illustrations of $\widetilde{C}_3$ and $\widetilde{C}_4$, respectively. Given a graph $H$ and a positive integer $k,$ the {\it $k$-blow-up of $H,$} denoted by $H^{(k)},$ is the graph obtained by replacing
every vertex of $H$ with $k$ different vertices where a copy of $u$ is adjacent to a copy of $v$ in the blow-up graph if and only if $u$ is adjacent to $v$ in $H$.
\begin{lemma}[Zhan \cite{Zhan2024}]\label{Zhan-triangle}
 Let $G$ be a $[p+2,p]$-graph of order $n$ with $\delta(G)\ge p\ge 2$ and $n\ge 2p+3.$ Then $G$ is triangle-free if and only if $p$ is even, $p\ge 6$ and $G=C_5^{(p/2)}$.
\end{lemma}
A {\it bowtie} is the graph consisting of two triangles with one vertex identified and an \emph{$\ell$-dumbbell} is a graph obtained by connecting two triangles with a path of length~$\ell$.
See Figure~\ref{Figureeight}$(c)(d)$ for a depiction.

\begin{lemma}\label{Lemma-[5,3]-diamond-house}
Let $G$ be a $[5, 3]$-graph of order at least 9 with $\delta(G) \geq 3$. Then $G$ contains either a $\widetilde{C}_3$ or a $\widetilde{C}_4$.
\end{lemma}
\begin{proof}[\bf Proof]
To the contrary, assume that $G$ is a $[5, 3]$-graph of order at least 9 with $\delta(G) \geq 3$ that contains neither a $\widetilde{C}_3$ nor a $\widetilde{C}_4$.

\begin{claim}\label{Claim-bowtie}
$G$ is bowtie-free.
\end{claim}

Otherwise, suppose that $G$ contains a bowtie $B$. For convenience, let the two triangles of $B$ be $xx_1x_2x$ and $xx_3x_4x$, where $x$ is the common vertex.  Since $\delta(G)\geq 3$ and $G$ is $\widetilde{C}_3$-free, there exist $y_1,y_2\in V(G)\setminus V(B)$ such that $x_1\sim y_1$ and $x_2\sim y_2$.   Since $G$ contains neither a $\widetilde{C}_3$ nor a $\widetilde{C}_4$, $G[\{x_1,y_1,y_2,x_3,x_4\}]$ has size at most 2, a contradiction.  This proves Claim~\ref{Claim-bowtie}.

\begin{claim}\label{Claim-dumbbell}
$G$ is $1$-dumbbell-free.
\end{claim}

Otherwise, suppose that $G$ contains a dumbbell $D$. For convenience, let the two triangles of $D$ be $x_1x_2x_3x_1$ and $y_1y_2y_3y_1$, where $x_1y_1$ is an edge. Since $G$ is $\widetilde{C}_3$-free and $\delta(G)\geq 3$, there exist two distinct vertices $x_2^{\prime},x_3^{\prime}\in V(G)\setminus V(D)$ such that $x_2\sim x_2^{\prime}$ and $x_3\sim x_3^{\prime}$. Since $G[\{x_2^{\prime},x_3^{\prime},x_1,y_1,y_2\}]$ has size at least 3, one has
$x_2^{\prime}\sim y_2$ or $x_3^{\prime}\sim y_2$.
 Without loss of generality, let $x_2^{\prime}\sim y_2$.
Note that $x_2^{\prime}$ cannot be adjacent to both $y_2$ and $y_3$ simultaneously, otherwise there exists a $\widetilde{C}_3$.
 Since
$G[\{x_1,y_1,y_3,x_2^{\prime},x_3^{\prime}\}]$ has size at least 3, we have $x_3^{\prime}\sim y_3$.
Since $\delta(G)\geq 3$ and $|V(G)|\geq9$, there exist a vertex $z\in V(G)\setminus (V(D)\cup\{x_2^{\prime},x_3^{\prime}\})$ such that $z\sim x_2^{\prime}$.
Since $G$ is $(\widetilde{C}_3,\widetilde{C}_4,\text{bowtie})$-free, $G[\{x_1,x_2,y_2,y_3,z\}]$ has size 2, a contradiction. This proves Claim~\ref{Claim-dumbbell}.

\medskip

By Lemma~\ref{Zhan-triangle}, $G$ contains a triangle $E$, say $v_1 v_2 v_3 v_1$. Since $\delta(G)\geq 3$ and $G$ is $\widetilde{C}_3$-free, there exist distinct $v_4,v_5,v_6\in V(G)\setminus V(E)$ such that $v_1\sim v_4$, $v_2\sim v_5$ and $v_3\sim v_6$. Since $\delta(G)\geq 3$ and $G$ is $(\widetilde{C}_3,\widetilde{C}_4)$-free, there exist $v_7,v_8\in V(G)\setminus (V(E)\cup\{v_4,v_5,v_6\})$ such that $\{v_7,v_8\}\subseteq N(v_4)$.
Since $G[\{v_1,v_5,v_6,v_7,v_8\}]$ has size at least 3, $|[\{v_5,v_6\},\{v_7,v_8\}]|\geq 3$. Without loss of generality, let $\{v_7,v_8\}\subseteq N(v_5)$ and $v_6\sim v_7$. Since $G[\{v_1,v_2,v_6,v_7,v_8\}]$ has size at least 3, it follows $v_6\sim v_8$.

\medskip

Denote $F=G[\{v_1,v_2,v_3,v_4,v_5,v_6,v_7,v_8\}]$. Since $|V(G)|\geq 9$ and $G$ is connected, there exists a vertex $v_9\in V(G)\setminus V(F)$ such that $v_9$ is adjacent to at least one vertex of $V(F)$. By symmetry, we can only consider the following three situations.
If $v_9\sim v_8$, then by Claim~\ref{Claim-dumbbell}, $v_9\notin N(\{v_4,v_5,v_6\})$, so $G[\{v_1,v_4,v_5,v_6,v_9\}]$ has size at most 2, a contradiction.
If $v_9\sim v_4$, then $v_9\notin N(\{v_7,v_8\})$ (Claim~\ref{Claim-dumbbell}), hence $G[\{v_2,v_3,v_7,v_8,v_9\}]$ has size 1, a contradiction.
If $v_9\sim v_1$, then $v_9\nsim v_4$ (Claim~\ref{Claim-bowtie}), so $G[\{v_3,v_4,v_5,v_6,v_9\}]$ has size 1, a contradiction.
This completes the proof of Lemma~\ref{Lemma-[5,3]-diamond-house}.
\end{proof}
\begin{lemma}\label{lem:distance}
Let $G$ be a connected $[p+2, p]$-graph, and let $S$ and $T$ be two disjoint vertex subsets of $G$ with $|S| \geq 2$ and $|T| \geq 2$. Then the shortest $(S,T)$-path in $G$ has length at most $p$.
\end{lemma}
\begin{proof}[\bf Proof]
	To the contrary, suppose that any $(S,T)$-path in $G$ has length at least $p+1$. Let $P=x_1x_2\ldots x_m$ be a shortest $(S,T)$-path, where $m\geq p+2$. Choose any vertices $u_1,u_2\in S$ and $w_1,w_2\in T$.  Then $G[\{u_1,u_2,w_1,w_2,x_3,x_4,\ldots,x_{p}\}]$ has size most $p-1$, a contradiction. This completes the proof of Lemma~\ref{lem:distance}.
\end{proof}

The central gadget we use in our proof is given by the following definition, which was initially introduced in \cite{Draganic2024}.
\begin{definition}[Dragani\'{c}-Correia-Sudakov \cite{Draganic2024}]
	A $(t,s,\ell)$-switch in $G$ is a subgraph $R$ which consists of a path $P=v_1,v_2,\ldots,v_{\ell+1}$ together with the vertex $y$ adjacent to vertices $t,t+1,\ldots,t+s$ with $t,s\geq 1.$ We also write $(t,\cdot)$-switch to denote a switch for which the $s$ is not specified and the length of $P$ is $\ell$.
\end{definition}

\begin{figure}[ht]
	\centering
	\begin{tikzpicture}[scale=1.05,
		main node/.style={circle,draw,color=black,fill=black,inner sep=0pt,minimum width=4.5pt}
		]
		\node[main node] (v1) at (2,0) [label=below:$v_{1}$]{};
		\node[main node] (vl1) at (12,0) [label=below:$v_{\ell+1}$]{};
		
		\node[main node] (vt1) at (5,0) [label=below:$v_{t}$]{};
		\node[main node] (vt2) at (6,0) [label=below:$v_{t+1}$]{};
		\node[main node] (vt3) at (7,0) [label=below:$v_{t+2}$]{};
		\node[main node] (vt4) at (8,0) [label=below:$v_{t+3}$]{};
		\node[main node] (vt5) at (9,0) [label=below:$v_{t+4}$]{};
		
		\node[main node] (y) at (7,1) [label=above:$y$]{};
		
		\draw[dashed,line width=1pt] (vt5) -- (vl1);
        \draw[dashed,line width=1pt] (v1) -- (vt1);
		\draw[line width=1pt] (y) -- (vt1);
        \draw[line width=1pt] (vt1) -- (vt2);
        \draw[line width=1pt] (vt2) -- (vt3);
        \draw[line width=1pt] (vt3) -- (vt4);
        \draw[line width=1pt] (vt4) -- (vt5);
		\draw[line width=1pt] (y) -- (vt2);
		\draw[line width=1pt] (y) -- (vt3);
		\draw[line width=1pt] (y) -- (vt4);
		\draw[line width=1pt] (y) -- (vt5);
	\end{tikzpicture}
	\caption{A $(t,4,\ell)$-switch.}\label{(t,4)-switch}
\end{figure}
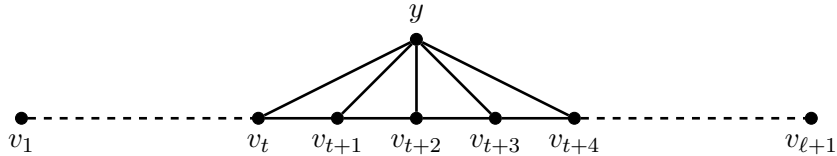

Figure \ref{(t,4)-switch} illustrates a  $(t,4,\ell)$-switch. If a $\widetilde{C}_{\ell}$ (not necessarily induced) is in a graph $G$ of order $n$ for $\ell< n-1$, then there is an edge between $\widetilde{C}_{\ell}$ and a vertex $v$ outside of $\widetilde{C}_{\ell}$, which evidently produces a $(t,1,\ell)$-switch whose path starts at $v$.

\section{Proof of Theorem~\ref{thm:5-3}}\label{sec4}

The aim of this section is to prove Theorem~\ref{thm:5-3}. By Lemma~\ref{Lemma-[5,3]-diamond-house}, $G$ contains a $\widetilde{C}_{3}$ or $\widetilde{C}_{4}$, hence Theorem~\ref{thm:5-3} can be deduced from the following theorem.

\begin{theorem}\label{thm:l-cycle-expansion}
Let $G$ be a $2$-connected $[5,3]$-graph of order $n\geq 9$ with $\delta(G)\geq 3$. If $G$ contains a $\widetilde{C}_{\ell}$ for some $\ell<n-1$, then $G$ also contains $\widetilde{C}_{\ell+1}$ or $\widetilde{C}_{\ell+2}$.

\end{theorem}
\begin{proof}[\bf Proof]
To the contrary, suppose that $G$ contains a $\widetilde{C}_{\ell}$ but contains neither a $\widetilde{C}_{\ell+1}$ nor a $\widetilde{C}_{\ell+2}$.  We fix such an $\ell$ until the end of the proof. Since $G$ contains a $\widetilde{C}_\ell$, it contains an $(t, s,\ell)$-switch for some integers $t$ and $s$. Let $R$ be a $(t, s, \ell)$-switch consisting of a path $P=v_1v_2\ldots v_{\ell+1}$ together with the vertex $y$. Furthermore, we assume that $t$ is minimized and $s$ is maximized with respect to $t$, and among all such choices, the smallest subscript of the neighbor of $v_{\ell+1}$ on the path $P$ is made as small as possible.

\begin{claim}\label{cliam:l-size}
$\ell\geq 4$.
\end{claim}
Otherwise, assume that $\ell\leq 3$. Note that $G$ contains a $\widetilde{C}_{3}$ or $\widetilde{C}_{4}$.
That is, $G$ contains a $\widetilde{C}_{3}$ and $\ell=3$. By the choice of $R$, we have that $R$ is a $(1,2,3)$-switch. Since $G$ contains neither $\widetilde{C}_{4}$ nor $\widetilde{C}_{5}$ and $\delta(G)\geq 3$,  $N_{\overline{R}}(v_4)\neq \emptyset$ and $N(\{y,v_1,v_2\})\cap N_{\overline{R}}[v_4]=\emptyset$. Let $v_5\in N_{\overline{R}}(v_4)$. Since $G$ is $2$-connected, there exists a $(\{v_1,v_2,y\},\{v_4,v_5\})$-path in $G-v_3$.  Let $P_{v_3}$ be such a shortest path. Since $[\{y,v_1,v_2\}, \{v_4,v_5\}]=\emptyset$, by a Similar analysis as Lemma~\ref{lem:distance}, $P_{v_3}$ has length at most $3$.
For convenience, let $P_{v_3} = u_1u_2\ldots u_k$ with $3\le k\le 4$. We consider the following two situations:
\begin{itemize}
\item $k=3$. Without loss of generality, let $u_1\in \{v_1,y\}$ and $u_3=v_5$. Then the edge $u_1u_2$ always creates a $\widetilde{C}_{5}$, a contradiction.
\item $k=4$. If $u_3\sim v_4$, then the situation is similar to the case $k=3$. Therefore, $u_4=v_5$.
We assert that $|N(u_2)\cap \{v_1,v_2,y\}|\geq 2$. Otherwise, the two vertices not adjacent to $u_2$, together with $\{u_2,v_4,v_5\}$, induce a subgraph of size at most $2$, a contradiction. On the one hand, if $v_1\sim v_3$, then this always creates a $\widetilde{C}_{4}$, a contradiction. On the other hand, if $v_1\sim u_2$, then this also always creates a $\widetilde{C}_{4}$, a contradiction. Therefore, $N(v_1)\cap \{u_2,v_3\}=\emptyset.$ If $u_2\sim v_3$, then the cycle $yu_2v_3v_2y$ together with $v_1$ forms a $\widetilde{C}_{4}$, a contradiction. If $v_3\sim v_5$, then the cycle $yu_2u_3v_5v_3y$ together with $v_2$ forms a $\widetilde{C}_{5}$, a contradiction. This implies that $G[\{v_1,v_3,v_4,u_2,u_4\}]$ has size at most $2$, a contradiction. This proves Claim~\ref{cliam:l-size}.
\end{itemize}

\begin{claim}\label{Claim-l-3-dumbbell}
    $G$ is $(\ell-3)$-dumbbell-free.
\end{claim}
Otherwise, assume that $G$ contains a $(\ell-3)$-dumbbell $D$. For convenience, suppose that the two triangles of $D$ are $v_1v_2v_3v_1$ and $v_{\ell}v_{\ell+1}v_{\ell+2}v_{\ell}$, and the path $P$ connecting them is $v_3\ldots v_{\ell}$. Furthermore, we consider the following ordering of its vertices: $v_1v_2\ldots v_{\ell+2}$. Since $G$ contains neither $\widetilde{C}_{\ell+1}$ nor $\widetilde{C}_{\ell+2}$, we have
$[\{v_1,v_2\},\,\{v_{\ell+1},v_{\ell+2}\}] = \emptyset$,
$N_P^-(\{v_1,v_2\}) \cap N(\{v_{\ell+1},v_{\ell+2}\}) = \emptyset$,
and $N_P^+(\{v_{\ell+1},v_{\ell+2}\}) \cap N(\{v_1,v_2\}) = \emptyset$.

\medskip

Now, we choose $v_i\in N_P(\{v_1,v_2\})$ such that $i$ is as large as possible. By symmetry, assume that $v_1\sim v_i$.  If $i< \ell$, then for every $r\in [i+1,\ell]$, $G[\{v_1,v_2,v_{\ell+1},v_{\ell+2},v_r\}]$ has size at least $3$, which implies $v_r \in N(\{v_{\ell+1},v_{\ell+2}\})$. For $i\ge 4$, since $G[\{v_1,v_2,v_{i-1},v_{\ell+1},v_{\ell+2}\}]$ has size at least $3$, $v_{i-1} \in N(\{v_{1},v_{2}\})$. For  $i-2\ge 3$, $G[\{v_1,v_2,v_{i-2},v_{\ell+1},v_{\ell+2}\}]$ has size at least $3$, which implies $v_{i-2} \in N(\{v_{1},v_{2}\})$.
 Repeating the process over and over again, it follows that $\{v_3,\ldots,v_i\}\subseteq N(\{v_1,v_2\})$.
Now, denote $S = \{v_1, v_2, \ldots, v_{i-1}\}$ and $T = \{v_{i+2}, \ldots, v_{\ell+2}\}$.  Furthermore, we define the set $X$ as follows:
\[
X =
\begin{cases}
	\{v_i\} & \text{if } i = \ell, \\
	\{v_i, v_{i+1}\} & \text{otherwise}.
\end{cases}
\]

\medskip

Suppose that $G-X$ is connected.  There exists an $(S, T)$-path in $G-X$. Without loss of generality, let $P_1$ be a shortest $(S, T)$-path. By Lemma~\ref{lem:distance}, $P_1$ has length at most $3$. For convenience, let $P_1 = u_1u_2\ldots u_s$ with $2\le s\le 4$.   We consider the following three situations:
\begin{itemize}
\item $k=2$. By symmetry, we may assume that $u_1^+\sim v_1$ and $u_2^-\sim v_{\ell+2}$.  This always creates a $\widetilde{C}_{\ell+1}$, see Figure \ref{fig-claim-3.1}$(a)$.
\item $k=3$. Since $G[\{v_1,v_2,u_2,v_{\ell+1},v_{\ell+2}\}]$ has size at least three, $N(u_2)\cap \{v_1,v_2,v_{\ell+1},v_{\ell+2}\}\neq \emptyset$. Without loss of generality, we may assume that $u_1=v_1$.  The edge  $u_2u_3$ always creates a  $\widetilde{C}_{\ell+1}$, see Figure \ref{fig-claim-3.1}$(b)$.
    \item $k=4$. Since $G[\{v_1,v_2,u_2,v_{\ell+1},v_{\ell+2}\}]$  and $G[\{v_1,v_2,u_3,v_{\ell+1},v_{\ell+2}\}]$  have size at least three, $N(u_2)\cap \{v_1,v_2,v_{\ell+1},v_{\ell+2}\}\neq \emptyset$ and $N(u_3)\cap \{v_1,v_2,v_{\ell+1},v_{\ell+2}\}\neq \emptyset$. Without loss of generality, we may assume that $u_1=v_1$ and $u_4=v_{\ell+2}$. This creates a $\widetilde{C}_{\ell+2}$, see Figure \ref{fig-claim-3.1}$(c)$.
\end{itemize}
 Therefore, $G-X$ is  not connected.  Then $i\neq \ell$. Note that $G$ is $2$-connected.
 Let $P_2$ be a shortest $(S, T)$-path in $G-v_i$, and let $P_3$ be a shortest $(S, T)$-path in $G-v_{i+1}$. For convenience,  we may assume that $P_2=u_1u_2\ldots u_k$ and   $P_3=w_1w_2\ldots w_r$.    By Lemma~\ref{lem:distance}, $\max\{k,r\}\leq 4$.  Since $G-X$ is not connected, $u_{k-1}=v_{i+1}$ and $w_{r-1}=v_{i}$. However, This always creates a  $\widetilde{C}_{\ell+1}$ or  $\widetilde{C}_{\ell+2}$  (see Figure \ref{fig-claim-3.1}$(d)(e)(f)$). This proves Claim~\ref{Claim-l-3-dumbbell}.

\begin{figure}[ht!]
\centering
    \includegraphics[width=0.95\linewidth]{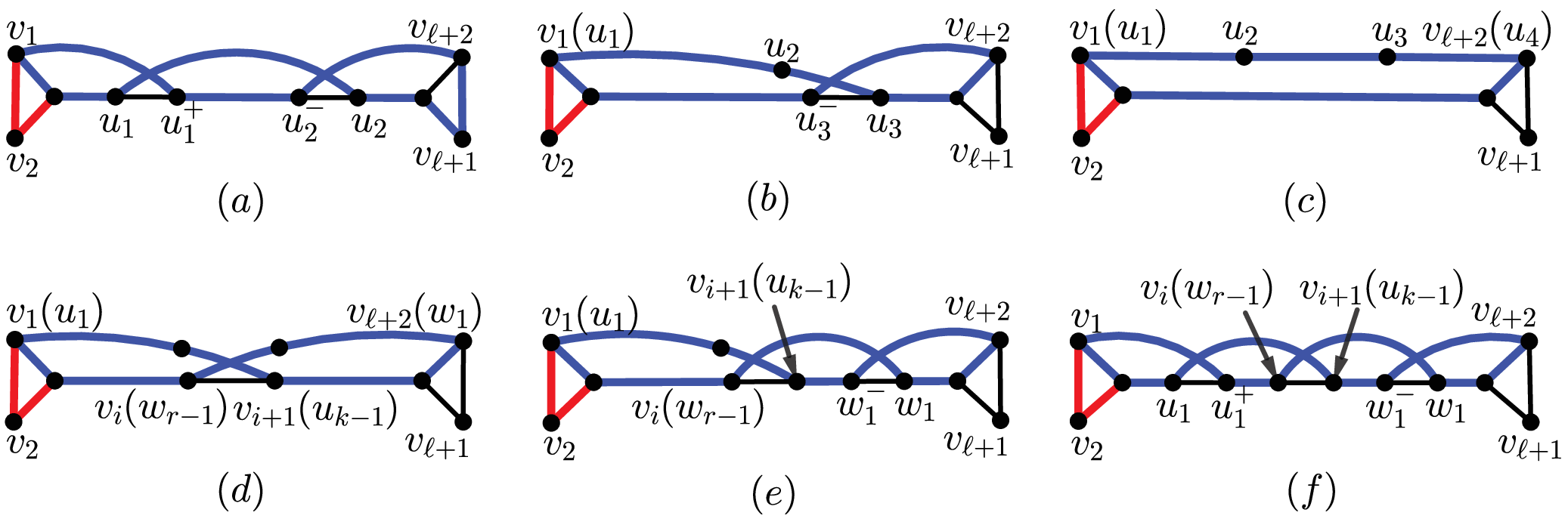}
\caption{Illustrations of Claim~\ref{Claim-l-3-dumbbell}.}
\label{fig-claim-3.1}
\end{figure}

\begin{claim}\label{Claim-l-2-dumbbell}
$G$ is $(\ell-2)$-dumbbell-free.
\end{claim}
Otherwise, assume that $G$ contains a $(\ell-2)$-dumbbell $D$. For convenience, suppose that the two triangles of $D$ are $v_1v_2v_3v_1$ and $v_{\ell+1}v_{\ell+2}v_{\ell+3}v_{\ell+1}$, and the path $P$ connecting them is $v_3\ldots v_{\ell+1}$. Furthermore, we consider the following ordering of its vertices: $v_1v_2\ldots v_{\ell+3}$. Note that $N_P^-(\{v_1,v_2\})\cap N(\{v_{\ell+1},v_{\ell+2}\})=\emptyset$ and $N_P^+(\{v_{\ell+1},v_{\ell+2}\})\cap N(\{v_1,v_2\})=\emptyset$. By Claim~\ref{Claim-l-3-dumbbell}, $G$ is $(\ell-3)$-dumbbell-free, thus $v_4\notin N(\{v_1,v_2\})$ and $v_{\ell}\notin N(\{v_{\ell+2},v_{\ell+3}\})$. Since $G[\{v_1,v_2,v_4,v_{\ell+2},v_{\ell+3}\}]$ has size at least $3$, it gives $v_4\in N(\{v_{\ell+2},v_{\ell+3}\})$.   Next, we consider the subgraph induced by $\{v_1,v_2,v_{5},v_{\ell+2},v_{\ell+3}\}$,
which has size at least $3$. This implies that
$v_{5}\in N(\{v_{\ell+2},v_{\ell+3}\})$. Repeating the process over and over again, it follows that $v_{\ell}\in  N(\{v_{\ell+2},v_{\ell+3}\})$, a contradiction. This proves Claim~\ref{Claim-l-2-dumbbell}.

\medskip

By an argument similar to that of Claim~\ref{Claim-l-2-dumbbell},
we obtain the following claim.
\begin{claim}\label{Claim-l-1-dumbbell}
$G$ is $(\ell-1)$-dumbbell-free.
\end{claim}

 From now on, we will differentiate between two situations.
\subsection{Triangle at the start: \texorpdfstring{$t=1$}{t=1}}
Since $G$ does not contain a $\widetilde{C}_{\ell+1}$,
We have $N_{\overline{R}}[v_{\ell+1}] \cap N(\{v_1,y\})=\emptyset$.
We begin the proof of this section with the following claim.
\begin{claim}\label{claim:neighbor:l+1}
 $N_{\overline{R}}(v_{\ell+1})=\emptyset$.
\end{claim}
Otherwise, suppose that $N_{\overline{R}}(v_{\ell+1})\neq \emptyset$. Next, we consider two situations.
\begin{itemize}
    \item $|N_{\overline{R}}(v_{\ell+1})|\geq 2$. Let $v_{\ell+2},v_{\ell+2}^{\prime}\in N_{\overline{R}}(v_{\ell+1})$. By Claims~\ref{Claim-l-2-dumbbell} and~\ref{Claim-l-1-dumbbell}, we have that $\{v_{\ell},v_{\ell+2},v_{\ell+2}^{\prime}\}$  is an independent set. Since $G[\{v_1,y,v_{\ell},v_{\ell+2},v_{\ell+2}^{\prime}\}]$ has size at least $3$, we have $\{y,v_1\}\subseteq N(v_{\ell})$. Moreover, as $G[\{v_1,y,v_{\ell-1},v_{\ell+2},v_{\ell+2}^{\prime}\}]$ has size at least $3$, it follows that $\{y,v_1\}\subseteq N(v_{\ell-1})$. Repeating this argument, we obtain
\[
\{y,v_1\}\subseteq \bigcap_{i=3}^{\ell} N(v_i).
\]
Denote $S=\{v_1,y,v_2,\ldots,v_{\ell-1}\}$ and $T=\{v_{\ell+1},v_{\ell+2},v_{\ell+2}^{\prime}\}$.
Since $G-v_{\ell}$ is connected, there exists a $(S,T)$-path $P_5$ in $G-v_{\ell}$. Without loss of
generality, let $P_5$ be a shortest $(S,T)$-path. By Lemma~\ref{lem:distance}, $P_5$ has length at most $3$. Similar to the proof of Claim~\ref{Claim-l-3-dumbbell}, the path $P_5$ always creates either $\widetilde{C}_{\ell+1}$ or $\widetilde{C}_{\ell+2}$, a contradiction.

\item $|N_{\overline{R}}(v_{\ell+1})|=1$. Let $v_{\ell+2}\in N_{\overline{R}}(v_{\ell+1})$. Since $\delta(G)\ge 3$, $N(v_{\ell+1}) \cap \{v_2,\ldots,v_{\ell-1}\}\neq \emptyset$. Let $v_i,v_j \in N(v_{\ell+1}) \cap \{v_2,\ldots,v_{\ell-1}\}$. Furthermore, we choose $v_i$ such that $i$ is as large as possible and $j$ is as small as possible.  By Claim~\ref{Claim-l-3-dumbbell}, we have that $i\leq \ell-2$. Since $G[\{v_1,y,v_{i+1},v_{\ell+1},v_{\ell+2}\}]$ has size at least $3$, $v_{i+1}\sim v_{\ell+2}$. Repeating this argument, we obtain
\[
\{v_{i+1},\ldots v_{\ell-1}\}\subseteq  N(v_{\ell+2}).
\]
If $i\neq \ell-2$, then this creates an $(\ell-2)$-dumbbell (See Figure~\ref{fig-claim-5}$(a)$), contradicting Claim~\ref{Claim-l-2-dumbbell}. Then $i= \ell-2$. Since $G$ contains  neither  $\widetilde{C}_{\ell+1}$ nor  $\widetilde{C}_{\ell+2}$,
\begin{flalign*}
    N(\{v_1,y\})\cap \{v_{\ell-1},v_{\ell},v_{\ell+1},v_{\ell+2}\}=\emptyset.
\end{flalign*}

Now, we assert that the neighbors of $v_{\ell+1}$ on $P$ are consecutive. Otherwise, without loss of generality, we may assume that $v_{j+1}\nsim v_{\ell+1}$. This implies that $G[\{v_1,y,v_{j+1},v_{\ell-1},v_{\ell+1}\}]$ has size at most $2$, a contradiction. In fact, $j=\ell-2$. Indeed, if $j \le \ell-3$, then $G[\{v_1,y,v_{\ell-2},v_{\ell},v_{\ell+2}\}]$ has size at least $3$, which implies $v_{\ell-2} \sim v_{\ell+2}$. This creates an $(\ell-2)$-dumbbell (See Figure~\ref{fig-claim-5}$(b)$), a contradiction. By the choice of $R$, we have that $N(v_{\ell-1})\cap \{y,v_1,\ldots,v_{\ell-3}\}=\emptyset$. In fact, $\{y,v_1,\dots,v_{\ell-3}\}$ is a clique. Indeed, otherwise, $G[\{y,v_1,\dots,v_{\ell-3}\}]$ contains a subgraph $H$ of order $3$ with size at most $2$, and then $G[V(H)\cup \{v_{\ell-1},v_{\ell+1}\}]$ also has size at most $2$, a contradiction. If $v_{\ell}$ has a neighbor $v_{\ell}'$ in $\overline{R}$, then by Claims~\ref{Claim-l-3-dumbbell} and~\ref{Claim-l-2-dumbbell}, $G[\{v_1,y,v_{\ell-1},v_{\ell+1},v_{\ell}'\}]$ has size exactly $1$, a contradiction. Since $\delta(G)\geq 3$, then $v_\ell\sim v_{\ell-2}$. By symmetry, we have that $v_{\ell+2}\sim v_{\ell-2} $ (see Figure~\ref{claim:neighbor:l+1}$(b)$). Let $S=\{v_1,y,v_2,\ldots,v_{\ell-3}\}$, $T=\{v_{\ell-1},v_{\ell},v_{\ell+1},v_{\ell+2}\}$, and $X=\{v_{\ell-3},v_{\ell-2}\}$. By an argument identical to that in the proof of Claim~\ref{Claim-l-3-dumbbell}, we obtain a $ \widetilde{C}_{\ell+1} $ or $ \widetilde{C}_{\ell+2} $, a contradiction. This completes the proof of Claim~\ref{claim:neighbor:l+1}.
\end{itemize}

\begin{figure}[ht!]
\centering
    \includegraphics[width=0.95\linewidth]{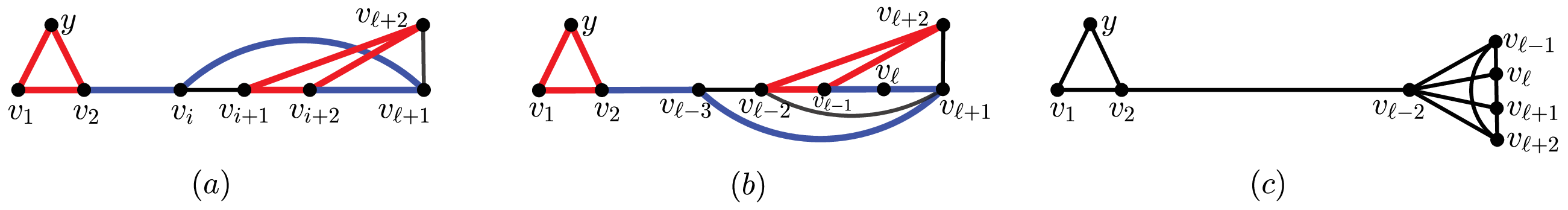}
\caption{Illustrations of Claim~\ref{claim:neighbor:l+1}.}
\label{fig-claim-5}
\end{figure}

By Claim~\ref{claim:neighbor:l+1}, we have that $|N_P(v_{\ell+1})|\geq 3$. Let $v_i,v_j\in N_P(v_{\ell+1})$ such that $i$ is as small as possible and $j$ is as large as possible. By Claim~\ref{Claim-l-3-dumbbell}, we have that $v_{\ell+1}\nsim v_{\ell-1}$. Hence $j\leq \ell-2$. Now, we assert that the neighbors of $v_{\ell+1}$ on $P[v_1,v_{\ell-2}]$ are consecutive. Otherwise, without loss of generality, we may assume that $v_{i+1}\nsim v_{\ell+1}$. This implies that $G[\{v_1,y,v_{i+1},v_{j+1},v_{\ell+1}\}]$ has size at most $2$, a contradiction. Next, we assert that $\{v_{i+1},\ldots,v_j\}\subseteq N(v_{j+1})$. Otherwise, let $v_r\in \{v_{i+1},\ldots,v_j\}$ such that $v_r\nsim v_{j+1}$. However, it follows that $G[\{v_1,y,v_{r},v_{j+1},v_{\ell+1}\}]$ has size at most $2$, a contradiction.  If $i\leq j-2$, then $v_{j-1}\sim v_{j+1}$. This creates a $(\ell-3)$-dumbbell, contradicting Claim \ref{Claim-l-3-dumbbell}. Hence, $i=j-1$. If $v_1\sim v_{\ell}$, then the cycle $$v_1\overrightarrow{P}[v_1,v_{j-1}]v_{j-1}v_{\ell+1}v_j\overrightarrow{P}[v_j,v_{\ell}]v_{\ell}v_1$$ together with $y$ forms a $\widetilde{C}_{\ell+1},$  a contradiction. Hence $v_{\ell}\nsim v_1$. By symmetry, $v_{\ell}\nsim y$. Since $G[\{v_1,y,v_{j+1},v_{\ell},v_{\ell+1}\}]$ has size at least $3$, $v_{j+1}\sim v_{\ell}$. Repeating this argument, we obtain
\begin{align*}
    \{v_{j+1},\ldots,v_{\ell-2}\} \subseteq N(v_{\ell}).
\end{align*}
However, this creates a $(\ell-3)$-dumbbell, a contradiction. Therefore, $j=\ell-2$. By the choice of $R$ and Claim~\ref{Claim-l-3-dumbbell}, we have that $\{v_{\ell-1},v_{\ell}\}\subseteq N( v_{\ell-3}).$ In fact, $\{y,v_1,\dots,v_{\ell-3}\}$ is a clique. Indeed, otherwise, $G[\{y,v_1,\dots,v_{\ell-4}\}]$ contains a subgraph $H$ of order $3$ with size at most $2$, and then $G[V(H)\cup \{v_{\ell-1},v_{\ell+1}\}]$ also has size at most $2$, a contradiction.  Denote $S=\{v_1,y,v_2,\ldots,v_{\ell-5}\}$, $T=\{v_{\ell-2},v_{\ell-1},v_{\ell},v_{\ell+1}\}$, and $X=\{v_{\ell-4},v_{\ell-3}\}$. By an argument identical to that in the proof of Claim~\ref{Claim-l-3-dumbbell}, we obtain a $ \widetilde{C}_{\ell+1} $ or $ \widetilde{C}_{\ell+2} $, a contradiction. This completes the proof for the case $t=1$.

\subsection{Triangle in the middle: \texorpdfstring{$t\geq 2$}{t>1}}
Note that $v_{\ell+1}$ has no neighbors outside $V(R)$. Indeed, if it had such a neighbor, we could add that neighbor to $R$ and remove $v_1$ from $R$, thus obtaining a $(t-1,s)$-switch, a contradiction. In addition, if $v_{\ell+1}\sim v_{\ell-1}$, then $\ell=3$, a contradiction. Therefore, $v_{\ell+1}\nsim v_{\ell-1}$. We begin this section with the following simple claim.
\begin{claim}\label{claim-L+1-neighbor-P}

$[N^+_{P[v_{t+1},v_{\ell}]}(v_{\ell+1}),\,\{v_1,\ldots,v_{t-1}\}]=\emptyset$.
\end{claim}
To the contrary, choose $v_j\in N_{P[v_{t+1},v_{\ell}]}(v_{\ell+1})$ and $v_k\in V(P)$, where $1\leq k\leq t-1$. Then the path $$v_{k}^+ \overrightarrow{P}[v_{k}^+,v_{j}] v_{j}v_{\ell+1}\overleftarrow{P}[v_{\ell+1},v_{j}^+]v_j^+v_k\overleftarrow{P}[v_{k},v_{1}]v_{1}$$ together with $y$ is a $(t-k,\cdot)$-switch. This proves Claim~\ref{claim-L+1-neighbor-P}.

In particular, $v_{\ell+1}$ has no neighbors $v_{k}$ with $k<t$.
\begin{claim}\label{claim:1-nsim-3}
$v_1\nsim v_3$.
\end{claim}
To the contrary, assume that $v_1 \sim v_3$. By the choice of $R$, this implies $t=2$ and $s=1$. By Claim~\ref{cliam:l-size}, $\ell\geq 4$. Thus, $v_{\ell+1}\nsim v_{\ell-1}$. Since $\delta(G)\geq 3$, $|N(v_{\ell+1})\cap \{v_t,\ldots,v_{\ell-2}\}|\geq 2$. Let $v_j \in N_{P[v_t,v_{\ell-2}]}(v_{\ell+1})$ such that $j$ is as large as possible. Hence, $j\geq 3$. Clearly, $\{v_1,y,v_{j+1},v_{\ell+1}\}$ is an independent set.
Suppose $j\geq 4$. Note that $s=1$. However, it follows that $G[\{v_1,y,v_4,v_{j+1},v_{\ell+1}\}]$  has size at most $2$, a contradiction.

Therefore, $j=3$ and $v_2\sim v_{\ell+1}$. Suppose that $v_1$ has a neighbor in $V(\overline{R})$, say $v_0$. Then $N(v_0)\cap \{v_4,v_{\ell+1}\}=\emptyset$. This implies that $G[v_0,v_1,y,v_4,v_{\ell+1}]$ has size at most $2$, a contradiction. Therefore, $v_{1}$ has no neighbors outside $V(R)$.

Suppose that $\ell\geq 6$.   Since $G[\{v_1,y,v_4,v_i,v_{\ell+1}\}]$ has size at least $3$ for any $i\in [5,\,\ell-1]$, it follows that $\{v_1,y\}\subseteq N(v_i)$.
  Now, we consider a new $(2,\cdot)$-switch $R'$ consisting of a path $v_1v_2v_3v_{\ell+1}\overleftarrow{P}[v_{\ell+1},v_4]v_4$ together with $y$.  By a similar analysis on $R'$, we obtain $\{v_1,y\}\subseteq N(v_{\ell})$.  However, the path
$v_3v_2v_1v_6\overrightarrow{P}[v_6,v_{\ell}]v_{\ell}yv_5v_4$
together with $v_{\ell+1}$ forms a $(1,\cdot)$-switch, a contradiction.

\medskip

Therefore, $\ell=5$.  Note that $N_{\overline{R}}(\{v_1,y,v_4,v_6\})=\emptyset$. Since $|V(G)|\geq 9$, there exists $z\in V(\overline{R})$. However, it follows that $G[\{z,v_1,y,v_4,v_6\}]$ has no edges, a contradiction. This proves Claim~\ref{claim:1-nsim-3}.

\begin{claim}\label{claim:1-consecuitive-neighbor}
$v_1$ can not be adjacent to two consecutive vertices in $V(P)$.
\end{claim}

Otherwise, suppose that there exists $v_i\in V(P)$ such that $\{v_i, v_{i+1}\}\subseteq N(v_1)$, and
further assume that $i$ is chosen to be as large as possible.  Then $v_1\nsim v_{i+2}$.
\begin{itemize}
    \item If $i\leq t-1$, then the path $v_2\overrightarrow{P}[v_2,v_{t}]v_{t}yv_{t+1}\overrightarrow{P}[v_{t+1},v_{\ell+1}]v_{\ell+1}$ together with $v_1$ is a $(t-1,\cdot)$-switch, a contradiction.

    \item If $t+1\leq i\leq \ell$.
 then the path $v_2\overrightarrow{P}[v_2v_{t+1}]v_{t+1}v_{1}v_{t+1}^+\overrightarrow{P}[v_{t+1}^+v_{\ell+1}]v_{\ell+1}$ together with $v_1$ is a $(t-1,\cdot)$-switch, a contradiction.
\end{itemize}

  Hence, $i=t$. If $s\geq 2$, then the path $$v_2\overrightarrow{P}[v_2,v_{t+1}]v_{t+1}yv_{t+1}^+\overrightarrow{P}[v_{t+1}^+,v_{\ell+1}]v_{\ell+1}$$ together with $y$ is a $(t-1,\cdot)$-switch, a contradiction. Thus, we have $s=1$. By Claim~\ref{claim:1-nsim-3}, we have that $t\geq 3$. By the choice of $i$ and the optimality of $R$, one has $y\notin N(\{v_t^-,v_{t+1}^{+}\})$ and $v_{t+1}^{+}\neq v_{\ell+1}$. Clearly, $v_t^{-}\notin N(\{ v_{\ell+1}, v_{t+1}^+\})$ and $v_1\notin  N(\{ y, v_{t+1}^+\})$. We have $v_1\nsim v_{t}^-$, otherwise the path $$v_2\overrightarrow{P}[v_2,v_{t}]v_{t}yv_{t+1}\overrightarrow{P}[v_{t+1},v_{\ell+1}]v_{\ell+1} $$ together with $v_1$ is a $(t-1,\cdot)$-switch. One has $v_{\ell+1}\nsim y$, otherwise the path $v_2,\overrightarrow{P}[v_{2},v_{\ell+1}],v_{\ell+1},y$ together with $v_{1}$ is a $(t-1,\cdot)$-switch.
Thus,  $G[\{v_1,v_t^{-},v_{t+1}^{+},v_{\ell+1},y\}]$ has size at most one, a contradiction.  This proves Claim~\ref{claim:1-consecuitive-neighbor}.

\begin{claim}\label{claim:v_1-neighbor}
$|N_{\overline{R}}(v_1)|\leq 1$.
\end{claim}
Otherwise, let $v_0,v_0'\in N_{\overline{R}}(v_1)$. Note that $v_{\ell+1}$ has no neighbors outside $R$. Since $\delta(G)\ge 3$, $v_{\ell+1}$ has a neighbor in $\{v_t,\ldots,v_{\ell-2}\}$. Let $v_j\in N(v_{\ell+1})$ be chosen with $j$ as large as possible. If $j\ne t$ or $s\ge 2$, then in order to avoid a $\widetilde{C}_{\ell+1}$ or $\widetilde{C}_{\ell+2}$ in $G$, we must have
\[
N(\{v_{\ell+1},v_{j+1}\})\cap \{v_0,v_0',v_1\}=\emptyset.
\]
This implies that  $G[\{v_0,v_0',v_1,v_{j+1},v_{\ell+1}\}]$ contains exactly two edges, a contradiction. Hence $j=t$ and $s=1$. Note that by Claim \ref{claim-L+1-neighbor-P}, $N(v_{\ell+1})\cap \{v_2,\ldots,v_{t-1}\}= \emptyset$. Since $\delta(G)\geq 3$, $y\sim v_{\ell+1}$. Now, we consider the subgraph induced by $\{v_0,v_0',v_1,v_{t+2},v_{\ell+1}\}$.
\begin{itemize}
    \item If $v_1\sim v_{t+2}$, then the path $v_tyv_{\ell+1}\overleftarrow{P}[v_{\ell+1},v_{t+2}]v_{t+2}v_1\overrightarrow{P}[v_1,v_{t-1}]v_{t-1}$ together with $v_{t+1}$ is a $(1,\cdot)$-switch, a contradiction.
    \item If $v_0\sim v_{t+2}$, then the path $v_tyv_{\ell+1}\overleftarrow{P}[v_{\ell+1},v_{t+2}]v_{t+2}v_0v_1\overrightarrow{P}[v_1,v_{t-2}]v_{t-2}$ together with $v_{t+1}$ is a $(1,\cdot)$-switch, a contradiction.
\end{itemize}
Hence $N(v_{t+2})\cap \{v_0,v_0',v_1\}=\emptyset$. Note that $N(v_{\ell+1})\cap \{v_0,v_0',v_1\}=\emptyset$. Since $G[\{v_0,v_0',v_1,v_{t+2},v_{\ell+1}\}]$ has size at least $3$, $v_{t+2}\sim v_{\ell+1} $. By the choice of $j$, we have that $t+2=\ell$.  If $t\geq 3$, then the path $v_{\ell}v_{\ell+1}v_t\overleftarrow{P}[v_t,v_{1}],v_1,v_0$ together with $y$ is a $(2,\cdot)$-switch, a contradiction. Then $t=2$.  Since $G[\{v_0,v_0',y,v_4,v_5\}]$ has size at least $3$, $N(y)\cap \{v_0,v_0'\}\neq \emptyset$. Without loss of generality, we may assume that $y\sim v_0$. By the choice of $R$, we have that $N(v_2)\cap \{v_0,v_0'\}=\emptyset$. One has $v_2\nsim v_4$, otherwise the path $v_3v_2v_1v_0y$ together with $v_4$ is a $(1,\cdot)$-switch. Therefore, $G[\{v_0,v_0',v_2,v_4,v_5\}]$ has size at most two, a contradiction. This proves Claim~\ref{claim:v_1-neighbor}.

\begin{figure}[ht!]
\centering
    \includegraphics[width=0.8\linewidth]{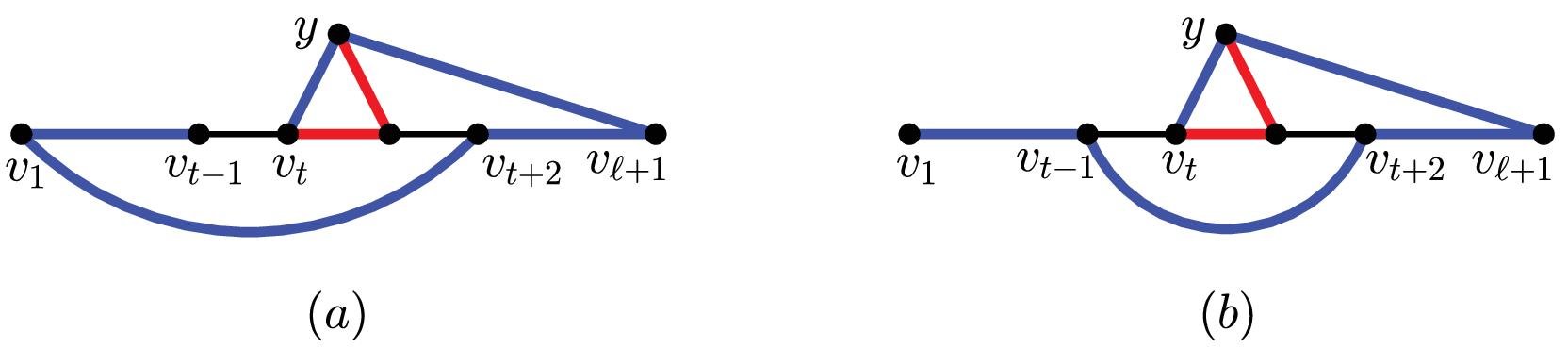}
\caption{Illustrations of Claim~\ref{claim:y-l+1}.}
\label{fig-claim-10}
\end{figure}

\begin{claim}\label{claim:y-l+1}
$v_{\ell+1}\nsim y$.
\end{claim}
Otherwise, assume that $v_{\ell+1}\sim y$. Note that $t\geq 2$. Observe that
\begin{flalign}\label{align:non-neighbor-pair}
   N(v_{t+2})\cap \{v_1,v_{t-1}\}=\emptyset.
\end{flalign}
indeed, the presence of these edges would create a $(1,\cdot)$-switch (see Figure \ref{fig-claim-10}).
We assert that $t\leq 3$. Otherwise, suppose that $t\geq 4$. If $v_1\sim v_{\ell}$, then the path $$v_{\ell+1}yv_{t+1}\overrightarrow{P}[v_{t+1},v_{\ell}]v_{\ell}v_1\overrightarrow{P}[v_1,v_{t-1}]v_{t-1}$$
 together with $v_t$ forms a $(2,\cdot)$-switch, a contradiction. If $v_1\sim v_{\ell-1}$, then path $$v_{\ell}v_{\ell+1}yv_{t+1}\overrightarrow{P}[v_{t+1},v_{\ell-1}]v_{\ell-1}v_1\overrightarrow{P}[v_1,v_{t-1}]v_{t-1}$$ together with $v_t$ forms a $(3,\cdot)$-switch, a contradiction.
If $v_{t-1}\sim v_{\ell}$, then the path $$v_{\ell+1}yv_{t+1}\overrightarrow{P}[v_{t+1},v_{\ell}]v_{t-1}\overleftarrow{P}[v_{t-1},v_1]v_1$$ together with $v_t$ forms a $(2,\cdot)$-switch, a contradiction. If $v_{t-1}\sim v_{\ell-1}$, then path $$v_{\ell}v_{\ell+1}yv_{t+1}\overrightarrow{P}[v_{t+1},v_{\ell-1}]v_{\ell-1}v_{t-1}\overleftarrow{P}[v_{t-1},v_1]v_1$$ together with $v_t$ forms a $(3,\cdot)$-switch, a contradiction. Note that $v_{\ell-1}\nsim v_{\ell+1}$ since $\ell \geq 4$ and $t \geq 2$. Since $G[\{v_1,v_{t-1},v_{\ell-1},v_{\ell},v_{\ell+1}\}]$ has size at least $3$, $v_1\sim v_{t-1}$. Let $v_{i}\in N(v_1)\cap \{v_3,\ldots,v_{\ell}\} $ with $i$ being as small as possible.
Then $i\leq t-1$. Clearly, $v_{i-1}\nsim v_{\ell+1}$. If $v_{i-1}\sim v_{\ell}$, then the path $$v_{\ell+1}yv_{t+1}\overrightarrow{P}[v_{t+1},v_{\ell}]v_{\ell}v_{i-1}\overrightarrow{P}[v_{i-1},v_{1}]v_{1}v_i\overrightarrow{P}[v_{i},v_{t-1}]v_{t-1}$$
 together with $v_t$ forms a $(2,\cdot)$-switch, a contradiction. If $v_{i-1}\sim v_{\ell-1}$, the path $$v_{\ell}v_{\ell+1}yv_{t+1}\overrightarrow{P}[v_{t+1},v_{\ell-1}]v_{\ell-1} v_{i-1}\overrightarrow{P}[v_{i-1},v_{1}]v_{1}v_i\overrightarrow{P}[v_{i},v_{t-1}]v_{t-1}$$
 together with $v_t$ forms a $(3,\cdot)$-switch, a contradiction. However, it follows that $G[v_1,v_{i-1},v_{\ell-1},v_{\ell},v_{\ell+1}]$ has size exactly $2$, a contradiction. Therefore, $t\leq 3$.

\medskip

First, suppose that $t=2$. By (\ref{align:non-neighbor-pair}) and Claim~\ref{claim:1-nsim-3}, it follows that $i\geq 5$ and $v_1\nsim y$.
\begin{itemize}
    \item $N_{\overline{R}}(v_1)\neq \emptyset$. Let $v_0\in N_{\overline{R}}(v_1)$. If $v_1\sim v_5$, then the path $v_3yv_{\ell+1}\overleftarrow{P}[v_{\ell+1},v_5]v_5v_1v_0$ together with $v_2$ forms $(1,\cdot)$-switch, a contradiction. By (\ref{align:non-neighbor-pair}),  $N(\{v_4,v_5\})\cap \{v_0,v_1\}=\emptyset$. This implies that $i\geq 6$. If $v_0\sim v_{i-1}$, then  the path
$$ v_2yv_{\ell+1}\overleftarrow{P}[v_{\ell+1},v_i]v_iv_1v_0v_{i-1}\overleftarrow{P}[v_{i-1},v_5]v_5 $$
forms a $(1,\cdot)$-switch with $v_3$, a contradiction. Hence $v_0\nsim v_{i-1}$.
 Since $G[\{v_0,v_1,v_4,v_{i-1},v_{\ell+1}\}]$ has size at least $3$, $\{v_{i-1},v_{\ell+1}\}\subseteq N(v_4)$. However, it follows that the path $$yv_2v_1v_i\overrightarrow{P}[v_i,v_{\ell+1}]v_{\ell+1}v_4\overrightarrow{P}[v_{4},v_{i-1}]v_{i-1}$$  forms a $(1,\cdot)$-switch with $v_3$, a contradiction.

\item $N_{\overline{R}}(v_1)=\emptyset$. Since $\delta(G)\geq 3$ and $v_1\nsim y$, $|N(v_1)\cap \{v_3,\ldots,v_{\ell}\}|\geq 2$. Let $v_{i'}\in N(v_1)\cap \{v_3,\ldots,v_{\ell}\} $ with $i'$ being as large as possible. Then $i'\neq i$. By Claim~\ref{claim:1-consecuitive-neighbor}, we have that $i'\geq i+2$. If $v_{i-1}\sim v_{i'-1}$, then the path $$yv_3\overrightarrow{P}[v_3,v_{i-1}]v_{i-1}v_{i'-1}\overleftarrow{P}[v_{i'-1},v_i]v_iv_1v_{i'}\overrightarrow{P}[v_{i'},v_{\ell+1}]v_{\ell+1}$$ forms a $(1,\cdot)$-switch, a contradiction. Therefore, $v_{i-1}\nsim v_{i'-1}$. Clearly, $N(v_{\ell+1})\cap \{v_{i-1},v_{i'-1}\}=\emptyset$. This implies that $\{v_1,v_{i-1},v_{i'-1},v_{\ell+1}\}$ is an independent set.  Since $G[\{v_1,y,v_{i-1},v_{i'-1},v_{\ell+1}\}]$ has size at least $3$, $\{v_{i-1},v_{i'-1},v_{\ell+1}\}\subseteq N(y)$. If $i\geq 6$, then $v_4\sim v_{\ell+1}$, since $G[\{v_1,v_4,v_{i-1},v_{i'-1},v_{\ell+1}\}]$ has size at least $3$ together with (\ref{align:non-neighbor-pair}).  However, it follows that  the path $$v_2yv_{i-1}\overleftarrow{P}[v_{i-1},v_4]v_4v_{\ell+1}\overleftarrow{P}[v_{\ell+1},v_{i'}]v_{i'}v_1v_i\overrightarrow{P}[v_i,v_{i'-1}]v_{i'-1}$$ forms a $(1,\cdot)$-switch, a contradiction. Hence $i=5$. This implies that $v_4yv_{\ell+1}\overleftarrow{P}[v_{\ell+1},v_5]v_5v_1v_2$ forms a $(1,\cdot)$-switch, a contradiction.

\end{itemize}

Next, we assume that $t=3$. This implies that $\ell\geq 5$. In fact, $\ell\geq 6$. Indeed, if $\ell=5$, then $N_{\overline{R}}(v_1)=\emptyset$. Since $\delta(G)\geq 3$ and $t=3$, $v_6\sim v_3$ and $\{y,v_4\}\subseteq N(v_1)$. However, it follows that the path $v_2v_1v_4v_5v_6v_3$ together with $y$ forms a $(2,\cdot)$-switch, a contradiction. Therefore, $\ell\geq 6$.

 If $v_1\sim v_{\ell}$, then the path $v_{\ell+1}yv_{4+1}\overrightarrow{P}[v_{4+1},v_{\ell}]v_{\ell}v_1v_2$
 together with $v_3$ forms a $(2,\cdot)$-switch, a contradiction. If $v_2\sim v_{\ell}$, then the path $v_{\ell+1}yv_{4+1}\overrightarrow{P}[v_{4+1},v_{\ell}]v_{\ell}v_2v_1$ together with $v_3$ forms a $(2,\cdot)$-switch, a contradiction. Therefore, $[\{v_1,v_2\},\,\{v_{\ell},v_{\ell+1}\}]=\emptyset$.
Since $\delta(G)\geq 3$, we have
$N(v_{\ell+1})\cap \{v_3,\ldots,v_{\ell-1}\}\neq \emptyset.$
Let $v_j\in N(v_{\ell+1})\cap \{v_3,\ldots,v_{\ell-1}\}$ with $j$ being chosen as large as possible.
\begin{itemize}
    \item $N_{\overline{R}}(v_2)\neq\emptyset$. Let $v_1'\in N_{\overline{R}}(v_2)$. Clearly, $N(\{v_5,v_{\ell+1}\})\cap \{v_1,v_1',v_2\}=\emptyset$ and $v_1'\nsim v_1$. Since $G[\{v_1,v_1',v_2,v_5,v_{\ell+1}\}]$ has size at least $3$, $v_5\sim v_{\ell+1}$. Then $j\geq 5$.  However, Claim~\ref{claim-L+1-neighbor-P} implies that $G[\{v_1,v_1',v_2,v_{j+1},v_{\ell+1}\}]$ has size at most $2$, a contradiction.

\item $N_{\overline{R}}(v_2)=\emptyset$. If $v_2\sim v_4$, then the path $v_1v_2v_4\overrightarrow{P}[v_4,v_{\ell+1}]v_{\ell+1}y$ together with $v_3$ forms a $(2,\cdot)$-switch, a contradiction. Since $\delta(G)\geq 3$, $N(v_2)\cap \{v_5,\ldots,v_{\ell-1}\}\neq \emptyset$. Let $v_i\in N(v_2)\cap\{v_5,\ldots,v_{\ell-1}\}$ with $i$ being chosen as small as possible. If $v_1\sim v_{i-1}$, then path $$yv_4\overrightarrow{P}[v_{4},v_{i-1}]v_{i-1}v_1v_2v_i\overrightarrow{P}[v_i,v_{\ell+1}]v_{\ell+1}$$ together with $v_3$ forms a $(1,\cdot)$-switch, a contradiction. If $v_{i-1}\sim v_{\ell}$, then the path $$v_{\ell+1}yv_4\overrightarrow{P}[v_4,v_{i-1}]v_{i-1}v_{\ell}\overleftarrow{P}[v_{\ell},v_i]v_iv_2v_1$$ together with $v_3$ forms a $(2,\cdot)$-switch, a contradiction. If $v_{\ell+1}\sim v_{i-1}$, then the path $$v_3\overrightarrow{P}[v_i,v_{i-1}]v_{i-1}v_{\ell+1}\overleftarrow{P}[v_{\ell+1},v_i]v_iv_2v_1$$ together with $y$ forms a $(1,\cdot)$-switch, a contradiction. This implies that $G[\{v_1,v_2,v_{i-1},v_{\ell},v_{\ell+1}\}]$ has size at most $2$, a contradiction.
\end{itemize}
This completes the proof of Claim~\ref{claim:y-l+1}.

\begin{claim}\label{claim:1-y}
    $v_1\nsim y$.
\end{claim}
Otherwise, assume that $v_1\sim y$. Then $t\geq 3$. Since $N_{\overline{R}}(v_{\ell+1})=\emptyset$ and $v_{\ell+1}\nsim v_{\ell-1}$, Claims~\ref{claim-L+1-neighbor-P} and~\ref{claim:y-l+1} imply that $N(v_{\ell+1})\cap \{v_{t+1},\ldots,v_{\ell-1}\}\neq \emptyset$. Let $v_j\in N(v_{\ell+1})\cap \{v_t,\ldots,v_{\ell-1}\}$ with $j$ chosen as large as possible. By Claim~\ref{claim-L+1-neighbor-P}, we have that $N(v_{j+1})\cap \{v_2,v_{t-1}\}=\emptyset$ and $v_2\nsim v_{\ell+1}$.

\medskip

Since $G[\{v_1,v_{t-1},v_{j+1},v_{\ell+1},y\}]$ has size at least $3$, $v_1\sim v_{t-1}$ and $y\sim v_{j+1}$.
    Suppose that $t\geq 4$.   If $v_2\sim v_{t-1}$, then the path
\[
    v_{t-2}\overleftarrow{P}[v_{t-2},v_1]\,v_1yv_t\overrightarrow{P}[v_t,v_{\ell+1}]\,v_{\ell+1}
\]
together with $v_{t-1}$ forms a $(t-3,\cdot)$-switch, a contradiction. Therefore,  $G[\{v_1,v_2,v_{t-1},v_{j+1},v_{\ell+1}\}]$ contains at most two edges, a contradiction.
Hence $t=3$. Suppose that $N_{\overline{R}}(v_2)\neq \emptyset$. Let $v_1'\in N_{\overline{R}}(v_2)$. Since $t=3$, $N(v_1')\cap\{v_1,v_3\}=\emptyset$. Clearly, $N(v_1')\cap \{v_{j+1},v_{\ell+1}\}=\emptyset$. This implies that $G[\{v_1,v_1',v_3,v_{j+1},v_{\ell+1}\}]$ has size at most two, a contradiction. Therefore, $N_{\overline{R}}(v_2)=\emptyset$.  Since $\delta(G)\geq 3$, $N(v_2)\cap \{v_4,\ldots,v_{\ell}\}\neq \emptyset$. Let $v_h\in N(v_2)\cap \{v_4,\ldots,v_{\ell}\}$.  If $h=4$, then the path $$v_1v_2v_4\overrightarrow{P}[v_4,v_{j}]v_{j}v_{\ell+1}\overleftarrow{P}[v_{\ell+1},v_{j+1}]v_{j+1}y$$ together with $v_{3}$ forms a $(2,\cdot)$-switch, a contradiction.  Then $h\geq 5$.

Let $v_r\in \{v_5,\ldots,v_{j}\}$ such that $\{v_5,\ldots,v_{r-1}\} \subseteq N(y)$ with $r$ being as large as possible. If $v_1\sim v_r$, then the path $$yv_{r-2}\overleftarrow{P}[v_{r-2},v_1]v_1v_r\overrightarrow{P}[v_r,v_{\ell+1}]v_{\ell+1}$$ together with $v_{r-1}$ forms a $(1,\cdot)$-switch, a contradiction. If $v_2\sim v_{r}$, then the path $$v_1yv_{r-2}\overleftarrow{P}[v_{r-2},v_2]v_2v_r\overrightarrow{P}[v_r,v_{\ell+1}]v_{\ell+1}$$ forms a $(2,\cdot)$-switch, a contradiction. Therefore, $N(v_r)\cap \{v_1,v_2\}=\emptyset$. Since $G[\{y,v_1,v_2,v_r,v_{\ell+1}\}]$ has size at least $3$, $v_{r}\sim v_{\ell+1}$ or $v_r\sim y$. In particular, if $v_r\sim y$, then $r=j-1$.

\begin{itemize}
    \item If $h\geq {r+1}$, then the path $v_1yv_{r-2}\overleftarrow{P}[v_{r-2},v_2]v_2v_h\overrightarrow{P}[v_h,v_{\ell+1}]v_{\ell+1}v_{r}\overrightarrow{P}[v_r,v_{h-1}]v_{h-1}$ with $v_{r-1}$ forms a $(2,\cdot)$-switch, a contradiction.
    \item If $5\leq h\leq  {r-1}$, then the path $v_1yv_{h-2}\overleftarrow{P}[v_{h-2},v_2]v_2v_{h}\overrightarrow{P}[v_h,v_{\ell+1}]v_{\ell+1}$  with $v_{h-1}$ forms a $(2,\cdot)$-switch, a contradiction.
\end{itemize}
This completes the proof of Claim~\ref{claim:1-y}.

\medskip

Suppose $t\geq 3$. Since $G[\{v_1,v_{t-1},y,v_{j+1},v_{\ell+1}\}]$ has size at least $3$, $v_{t-1}\sim v_{j+1}$. However, it follows that the path $yv_{t+1}\overrightarrow{P}[v_{t+1},v_j]v_jv_{\ell+1}\overleftarrow{P}[v_{\ell+1},v_{j+1}]v_{j+1}v_{t-1}\overleftarrow{P}[v_{t-1},v_1]v_1$ together with $v_{t}$ forms a $(1,\cdot)$-switch, a contradiction. Therefore, $t=2$. By Claim~\ref{claim:v_1-neighbor},
$|N_{\overline{R}}(v_1)|\leq 1$. We complete the proof of Theorem~\ref{thm:l-cycle-expansion} by analyzing the following cases.

\begin{case}
    $|N_{\overline{R}}(v_1)|=1$.
\end{case}
Let $v_0\in N_{\overline{R}}(v_1)$. Clearly, $N(v_0)\cap \{v_2,v_{\ell+1}\}=\emptyset$. By Claim~\ref{claim-L+1-neighbor-P} and \ref{claim:y-l+1}, we have that $|N(v_{\ell+1})\cap \{v_t,\ldots,v_{\ell-2}\}|\geq 2$. Let $v_j\in N(v_{\ell+1})\cap \{v_t,\ldots,v_{\ell-2}\}$ with $j$ as large as possible. Then $j\geq t+1$. By Claim~\ref{claim:1-nsim-3} and $\delta(G)\geq 3$, we have that $N(v_1)\cap \{v_4,\ldots,v_\ell\}\neq \emptyset$. Let $v_i\in N(v_1)\cap \{v_4,\ldots,v_\ell\}$ with $i$ as small as possible.

\begin{itemize}
    \item $i= j+2 $.  Then the path $$v_2\overrightarrow{P}[v_2,v_j]v_jv_{\ell+1}\overleftarrow{P}[v_{\ell+1},v_{j+2}]v_{j+2}v_1v_0$$ together with $y$ forms a $(1,\cdot)$-switch, a contradiction.
    \item $i\geq j+3$.  Clearly, $\{v_0,v_1,v_{j+1},v_{\ell+1}\}$ is an independent set and $v_{i-1}\nsim v_{\ell+1}$. By Claim~\ref{claim:1-consecuitive-neighbor}, we have that $v_1\nsim v_{i-1}$.  Since $G[\{v_0,v_1,v_{i-1},v_{j+1},v_{\ell+1}\}]$ has size at least $3$, $v_0\sim v_{i-1}$. However, it follows  the path $$v_2v_3\overrightarrow{P}[v_3,v_{i-1}]v_{i-1}v_0v_1v_i\overrightarrow{P}[v_i,v_{\ell+1}]v_{\ell+1}$$ together with $y$ forms a $(1,\cdot)$-switch, a contradiction.
    \item $i\leq j$. Recall that $i\geq 4$. Clearly, $v_{i-1}\nsim v_{j+1}$. However, it follows that $G[\{v_0,v_1,v_{i-1},v_{j+1},v_{\ell+1}\}]$ has size at most $2$, a contradiction.

\end{itemize}

\begin{case}
    $N_{\overline{R}}(v_1)= \emptyset$.
\end{case}
By  Claim~\ref{claim:1-y} and $\delta(G)\geq 3$, we have that $|N(v_1)\cap\{v_4,\ldots,v_{\ell}\}|\geq 2 $. Let $v_i,v_{i'}\in N(v_1)\cap\{v_4,\ldots,v_{\ell}\}$ with $i$ chosen as small as possible and $i'$ as large as possible.  By Claim~\ref{claim:1-consecuitive-neighbor}, we have that $i'\geq i+2$. Note that $G[\{v_1,v_{i-1},v_{i'-1},v_{\ell+1}\}]$ has size at most $1$, possibly with $v_{i-1}\sim v_{i'-1}$.
 Furthermore, let $v_j\in N(v_{\ell+1})\cap \{v_2,\ldots ,v_{\ell-2}\}$ with $j$ as large as possible.
\begin{itemize}
    \item $j\geq i'$. It is obvious that  $N(v_{j+1})\cap \{v_1,v_{i-1},v_{i'-1},v_{\ell+1}\}=\emptyset$. However, it follows that $G[\{v_1,v_{i-1},v_{i'-1},v_{\ell+1},v_{j+1}\}]$ has size at most $1$, a contradiction.
    \item $j=i'-2$. That is, $j+1=i'-1$. Now, $\{v_1,v_{i-1},v_{j+1},v_{\ell+1}\}$ is an independent set. However, by Claims~\ref{claim:y-l+1} and~\ref{claim:1-y}, it follows that
$G[\{v_1,v_{i-1},v_{j+1},v_{\ell+1},y\}]$ has size at most $2$, a contradiction.
   \item $i\leq j\leq i'-3$. Now, $G[\{v_1,v_{i-1},v_{i'-1},v_{\ell+1},v_{j+1}\}]$ has size at most $2$, a contradiction.
   \item $j=i-2$. That is, $j+1=i-1$. If $v_{j+1}\sim v_{i'-1}$, then the path $$v_2\overrightarrow{P}[v_2,v_j]v_jv_{\ell+1}\overleftarrow{P}[v_{\ell+1},v_{i'}]v_{i}v_1v_{i}\overrightarrow{P}[v_{i},v_{i'-1}]v_{i'-1}$$ together with $y$ forms a $(1,\cdot)$-switch, a contradiction. Therefore, $v_{j+1}\nsim v_{i'-1}$. However, by Claims~\ref{claim:y-l+1} and~\ref{claim:1-y}, it follows that
$G[\{v_1,v_{i'-1},v_{j+1},v_{\ell+1},y\}]$ has size at most $2$, a contradiction.
   \item $j\leq i-3$. If $v_{j+1}\sim v_{i'-1}$, then the path $$v_2\overrightarrow{P}[v_2,v_j]v_jv_{\ell+1}\overleftarrow{P}[v_{\ell+1},v_{i'}]v_{i'}v_1v_i\overleftarrow{P}[v_i,v_{j+1}]v_{j+1}v_{i'-1}\overleftarrow{P}[v_{i'-1},v_{i+1}]v_{i+1}$$ together with  $y$ forms a $(1,\cdot)$-switch, a contradiction. Therefore, $v_{j+1}\nsim v_{i'-1}$. However, it follows that  $G[\{v_1,v_{i-1},v_{i'-1},v_{\ell+1},v_{j+1}\}]$ has size at most $2$, a contradiction.

\end{itemize}
 This completes the proof of Theorem~\ref{thm:l-cycle-expansion}.
\end{proof}

\section{Concluding remarks}\label{sec5}
This paper investigates the pancyclicity of $[5,3]$-graphs and confirms Zhan’s conjecture for the case $p=3$. The proof method presented here may potentially be extended to tackle the remaining cases $p=4$ and $p=5$, possibly with additional new ideas.

\section*{Acknowledgments}
The authors are grateful to Professor Xingzhi Zhan for proposing this interesting conjecture and for his continuous support. Feng Liu is supported by National Key R\&D Program of China under Grant No. 2022YFA1006400, National Natural Science Foundation of China (Nos. 12571376) and Shanghai Municipal Education Commission (No. 2024AIYB003). Hongxi Liu was supported by the National Natural Science Foundation of China (Grant No. 12271170) and the Science and Technology Commission of Shanghai Municipality (Grant No. 22DZ2229014).

\section*{Declaration}

\noindent$\textbf{Conflict~of~interest}$
The authors declare that they have no known competing financial interests or personal relationships that could have appeared to influence the work reported in this paper.
	
\noindent$\textbf{Data~availability}$
Data sharing not applicable to this paper as no datasets were generated or analysed during the current study.

\end{document}